\documentclass[11pt]{amsart}
\pdfoutput=1
\usepackage[utf8]{inputenc} 
\usepackage[T1]{fontenc}
\usepackage{pgfplots}
\pgfplotsset{ compat=1.18}
\usepgfplotslibrary{fillbetween}

\usepackage{subcaption}
\usepackage{amsmath, amsthm, amssymb, enumerate, xcolor, tikz-cd, listings, MnSymbol, textcomp}
\usepackage{tikz}
\usepackage{float}
\usepackage[colorlinks=true, linkcolor=red, citecolor=blue, urlcolor=cyan]{hyperref}

\usepackage{fullpage}

\newcommand{\makeheading}[1]%
{\hspace*{-\marginparsep minus \marginparwidth}%
	\begin{minipage}[t]{\textwidth}%
		{\large \bfseries #1}\\[-0.15\baselineskip]%
		\rule{\columnwidth}{1pt}%
\end{minipage}}

\usepackage{mathtools}

\usepackage{amsmath, amsthm, amssymb, enumerate, hyperref, xcolor, tikz-cd, listings,hyperref,comment}

\usepackage{parskip}

\newcommand{\R}{\mathbb{R}}

\newcommand{\T}{\mathbb{T}}

\renewcommand{\hat}{\widehat}

\usepackage{soul}
\usepackage{cancel}

\theoremstyle{plain}
\newtheorem{theorem}{Theorem}
\newtheorem{theoremA}{Theorem}

\newtheorem{lemma}[theorem]{Lemma}

\newtheorem{proposition}[theorem]{Proposition}

\theoremstyle{definition}

\newtheorem*{definitionN}{Definition}

\renewcommand{\phi}{\varphi}

\usepackage{tikz}
\renewcommand{\epsilon}{\varepsilon}

\def\semicolon{;}
\def\applytolist#1{
	\expandafter\def\csname multi#1\endcsname##1{
		\def\multiack{##1}\ifx\multiack\semicolon
		\def\next{\relax}
		\else
		\csname #1\endcsname{##1}
		\def\next{\csname multi#1\endcsname}
		\fi
		\next}
	\csname multi#1\endcsname}

\def\calc#1{\expandafter\def\csname c#1\endcsname{{\mathcal #1}}}
\applytolist{calc}QWERTYUIOPLKJHGFDSAZXCVBNM;
\def\bbc#1{\expandafter\def\csname bb#1\endcsname{{\mathbb #1}}}
\applytolist{bbc}QWERTYUIOPLKJHGFDSAZXCVBNM;
\def\bfc#1{\expandafter\def\csname bf#1\endcsname{{\mathbf #1}}}
\applytolist{bfc}QWERTYUIOPLKJHGFDSAZXCVBNM;
\def\sfc#1{\expandafter\def\csname s#1\endcsname{{\sf #1}}}
\applytolist{sfc}QWERTYUIOPLKJHGFDSAZXCVBNM;
\def\fc#1{\expandafter\def\csname f#1\endcsname{{\mathfrak #1}}}
\applytolist{fc}QWERTYUIOPLKJHGFDSAZXCVBNM;

\title{Local smooth rigidity of Anosov \\diffeomorphisms in $\mathbb{T}^{3}$}

\author{James Marshall Reber}
\address{\parbox{\linewidth}{({\normalfont{J. Marshall Reber}}) \\ Department of Mathematics \\ University of Chicago \\ Chicago, IL 60637, USA \\[-.5em] }}
\email{jmarshallreber@uchicago.edu}

\author{Sebasti\'an Pavez-Molina}
\address{\parbox{\linewidth}{({\normalfont{S. Pavez-Molina}}) \\ Department of Mathematics \\ The Pennsylvania State University \\ State College, PA 16802, USA \\[-.5em] }}
\email{sap5974@psu.edu}

\begin{document}

\begin{abstract}
    Given a $C^0$ conjugacy between two Anosov diffeomorphisms, the matching periodic data problem asks whether this conjugacy is smooth provided spectral data of the diffeomorphisms match at periodic points. We show that if the two $C^0$ conjugate diffeomorphisms on $\mathbb{T}^3$ are sufficiently close to a hyperbolic linear automorphism with a pair of complex conjugate eigenvalues, then the conjugacy must be smooth. In particular, we have that in a neighborhood of a hyperbolic toral automorphism, matching periodic data implies that the conjugacy is $C^{1+\text{H\"older}}$.
\end{abstract}

\maketitle
\section{Introduction}

Let $M$ be a compact smooth manifold equipped with a Riemannian metric and let $f : M \rightarrow M$ be a diffeomorphism. 
The diffeomorphism $f$ is \emph{Anosov} if there exists a $Df$-invariant splitting of the tangent bundle into stable and unstable bundles such that $Df$ exponentially contracts vectors in the stable bundle and exponentially expands vectors in the unstable bundle, and it is \emph{transitive} if there exists a dense orbit.
For example, if $M = \mathbb{T}^n$, an automorphism in $\text{SL}(n, \mathbb{Z})$ whose spectrum does not intersect the unit sphere is a transitive Anosov diffeomorphism; such maps are called \emph{hyperbolic toral automorphisms}. 

Anosov diffeomorphisms are the prototypical example of chaotic dynamical systems, meaning that orbits are sensitive to small changes in initial data. 
Thanks to this, Anosov diffeomorphisms satisfy \emph{structural stability}: 
\begin{quote}
\emph{If $f$ is a $C^1$-Anosov diffeomorphism, then for every $C^1$-close diffeomorphism $g$ there is a unique homeomorphism $h : M \rightarrow M$ which is close to the identity and which conjugates the dynamical systems.}
\end{quote} 
Taking advantage of the dynamics, one can show that such a homeomorphism must be bi-H\"older.
A natural question is what conditions force this conjugating homeomorphism to have higher regularity.
One necessary condition for $C^1$-regularity of the conjugating map $h$ is that the diffeomorphisms have \emph{matching periodic data} in the following sense.
\begin{definitionN}
 Let $f, g: M \to M$ be two diffeomorphisms conjugated by a map $h: M \to M$. We say that $f$ and $g$ have \emph{matching periodic data} if for every periodic point $p \in M$ with $f^n(p) = p$, there exists a linear map $C_p: T_p M \to T_{h(p)} M$ such that $D_p f^n = C_p^{-1} (D_{h(p)} g^n) C_p$. 
\end{definitionN}
Our motivation is to understand when this necessary condition is sufficient: this is called the \emph{matching periodic data problem}. Note that it is not typical for a pair of Anosov diffeomorphisms to have matching periodic data -- it is possible to perturb a diffeomorphism in order to construct a $C^0$-conjugacy which does not have matching periodic data. In the case where $M = \mathbb{T}^2$, the matching periodic data problem was solved by de la Llave, Marco, and Moriy\'on \cite{de1987invariants, de1992smooth, marco1987invariants}, with the main technique involving the study of the cohomological equation and measure rigidity.
When $M = \mathbb{T}^n$ with $n \geq 3$, the problem becomes much more difficult due to the fact that one of the stable or unstable bundles has dimension at least two.

Progress towards the matching periodic data problem in the case where $n \geq 3$ has mostly been made towards the \emph{local matching periodic data problem}, which involves diffeomorphisms that are in a sufficiently small neighborhood of a hyperbolic toral automorphism. This is of particular interest when $M = \mathbb{T}^3$. When the spectrum of $L$ is real, Gogolev and Guysinsky showed that local matching periodic data implies that the conjugacy is $C^{1+\text{H\"older}}$ \cite[Theorem 1]{gogolev2008c}. In the case where the spectrum of $L$ has a pair of complex conjugate eigenvalues, Kalinin and Sadovskaya showed that local matching periodic data implies that the conjugacy is $C^{1+\text{H\"older}}$, provided one of the maps is linear \cite[Corollary 2.6]{kalinin2009anosov}. One can also interpret their argument as handling the case where one of the maps has constant periodic data. A bootstrapping argument by Gogolev shows that in the complex case, the conjugacy is smooth \cite{gogolev2017bootstrap}. More recently, progress has been made by Gogolev and Rodriguez Hertz for the non-algebraic case, i.e., when the maps do not necessarily have constant periodic data.  They showed that a dichotomy holds: if the periodic data of $f$ and $g$ are non-constant, then either the conjugacy is smooth or the SRB measure of one of the diffeomorphisms must coincide with the measure of maximal entropy \cite[Theorems 1.2 and 1.7]{GRH}. Combined with the previous results, this nearly establishes local periodic data rigidity. The main purpose of this paper is to handle the remaining case:

\begin{theoremA}\label{thm:main}
 Let $L: \mathbb{T}^{3} \to \mathbb{T}^{3}$ be a hyperbolic automorphism with a pair of non-real complex conjugate eigenvalues and let $r \geq 2$. There is a $C^{1}$-neighborhood $\mathcal{U}$ of $L$ such that if $f,g \in \mathcal{U}$ are any pair of $C^{r}$-Anosov diffeomorphisms which have non-constant matching periodic data and which satisfy the condition that the SRB measure and the measure of maximal entropy of $f$ coincide, then $f$ and $g$ are $C^{r_*}$-conjugate, where $r_*=r$ if $r$ is not an integer and $r_*=r-1+
 \text{Lip}$ if $r$ is an integer.
\end{theoremA}

Combining this with the work of Gogolev, Guysinsky, Kalinin, and Sadovskaya, we deduce the following.

\begin{theoremA} \label{thm:main2}
 Let $L: \mathbb{T}^{3} \to \mathbb{T}^{3}$ be a hyperbolic automorphism. There is a $C^{1}$-neighborhood $\mathcal{U}$ of $L$ such that for any pair of $C^{2}$-Anosov diffeomorphisms $f,g \in \mathcal{U}$, we have that $f,g$ are $C^{1+\text{H\"older}}$ conjugate provided that they have matching periodic data.
\end{theoremA}

We note that it is still an interesting problem as to what happens with the regularity of the conjugacy when the spectrum is real and there is non-constant matching periodic data. We heavily rely on the fact that there are a pair of complex conjugate eigenvalues in order to get smoothness of our conjugacy, and so \cite[Problem 1.6]{GRH} is still open (see also \cite[Remark after Theorem 1]{gogolev2008c} and \cite[Problem on page 2]{gogolev2017bootstrap}). We also briefly mention that both the higher dimensional and global problems are much more complicated. 

\begin{definitionN}
	Let $f : M \rightarrow M$ be a Anosov diffeomorphism and let $k \geq 0$. We say that $f$ is \emph{$C^k$ periodic data rigid} if any $g$ which is $C^0$-conjugate to $f$ and which has the same periodic data as $f$ must be $C^k$ conjugate to $f$.
\end{definitionN}

 It was recently established by DeWitt and Gogolev that every Anosov automorphism of $\mathbb{T}^3$ is smoothly periodic data rigid \cite[Theorem 1.1]{dewitt2024dominated} . 
Motivated by this and Theorem \ref{thm:main2}, we ask the following.

{\bf \noindent Question.} Are all Anosov diffeomorphisms on $\mathbb{T}^3$ $C^\infty$ periodic data rigid?

In higher dimensions, a classical counterexample by de la Llave shows that one should not expect even local matching periodic data rigidity to hold for $\mathbb{T}^n$ with $n \geq 4$ \cite{de1992smooth} (see also \cite[Theorem B]{gogolev2008smooth}).
In spite of this counterexample, a number of things can still be said. For example, it was recently shown by Kalinin, Sadovskaya, and Wang that for irreducible hyperbolic automorphisms where no three of the eigenvalues have the same modulus are periodic data rigid \cite[Corollary 1.4]{kalinin2024global}. In light of this, it is natural to ask how far one can push the methodology here and in \cite{GRH} to higher dimensions (see \cite[Theorem 1.7]{GRH} for some discussion on higher dimensions).

\subsection*{Organization of the paper} In Section \ref{sec:preliminaries}, we recall the basic tools on Anosov diffeomorphisms, linear cocycles, normal forms, and holonomies that will be used throughout the paper. A central ingredient in our approach is the theory of holonomies of fiber bunched cocycles introduced by Bonatti, G\'omez-Montt and Viana \cite{BonattiVianaGomez2003} and later refined by Sadovskaya, Kalinin \cite{kalinin2013cocycles,sadovskaya2015cohomology}. In Section \ref{sec:pcf} we leverage $C^{1}$-closeness of the diffeomorphism to the linear model in order to deduce regularity of holonomies of the unstable derivate cocycle.

Next, inspired by the work of \cite{GRH}, we use these holonomies to construct the so-called \emph{quadrilateral holonomy map} in Section \ref{sec:pcf}. The main goal is to show that this quadrilateral holonomy map is a smooth structure that is preserved by the conjugacy. This, in turn, implies that the conjugacy itself is smooth. The details of how the matching of these maps yields Theorem \ref{thm:main} are provided in Section \ref{sec:proof}. Therefore, it remains to establish that these maps match under the conjugacy.

The fact that the SRB and MME's align will allow us to work with the unstable derivatives as $
\text{SL}^{\pm}(2,\mathbb{R})$ cocycles. In particular, we discuss in Section \ref{sec:parry} about how the matching periodic data assumption turns into a matching trace assumption for $\text{SL}^\pm(2,\mathbb{R})$-cocycles, via the introduction of the \textit{Parry monoid} of a fiber bunched cocycle. This algebraic structure will serve as our main tool for understanding the behavior of holonomies along $us$-simple loops. 

Inspired by the recent work of Ceki\'{c} and Lefeuvre in \cite{cekic2025holonomy}, we show in Theorem \ref{thm:alternativetheorem} that this matching trace assumption leads to a dichotomy: either the cocycles have to be conjugate, and hence one can show that the smooth structures are matched by the conjugacy, or one of the cocycles is an almost coboundary. Going back to the matching periodic data problem, one of the cocycles being an almost coboundary implies that we are in the case where one of the diffeomorphisms has constant periodic data, and thus our assumptions force the smooth structure to be matched.

\subsection*{Acknowledgements}
The authors would like to thank Federico Rodriguez Hertz for suggesting the problem and many useful discussions. The authors would also like to thank Jairo Bochi and Jon DeWitt for discussions early on, and Andrey Gogolev, Amie Wilkinson, Victoria Sadovskaya, Boris Kalinin and Javier Echevarr\'ia Cuesta for feedback on early drafts. The first author was supported by the
National Science Foundation under award No.\ DMS-2503020.

\section{Preliminaries} \label{sec:preliminaries}

Throughout, let $\mathbb{T}^3$ be equipped with some Riemannian metric $\|\cdot\|$.

\subsection{Anosov diffeomorphisms}
We follow \cite{katok1995introduction, bowen1975equilibrium}. As described in the introduction, we say that $f$ is \emph{Anosov} if there is a $Df$-invariant splitting of the tangent bundle $TM = E^s \oplus E^u$ along with positive continuous functions $\mu, \nu < 1$ such that for all $x \in M$ and all unit vectors $v^s \in E^s(x)$, and $v^u \in E^u$, we have
\[ \|D_xf(v^s)\| < \nu(x) < \mu(x)^{-1} < \|D_xf(v^u)\|.\] 
From now on, $f : \mathbb{T}^3 \rightarrow \mathbb{T}^3$ will be a $C^r$ transitive Anosov diffeomorphism with $r \geq 2$. It will also be convenient to adopt the following notation for $n \in \mathbb{Z}$ and $x \in \mathbb{T}^3$:
\[ \nu(x,n) \coloneq \begin{dcases} \nu(f^{n-1}(x)) \cdots \nu(x) &\text{if } n \geq 1, \\ 1 & \text{if } n = 0, \\ (\nu(f^n(x),-n))^{-1} & \text{if } n \leq -1,  \end{dcases}\]
and we use similar notation for $\mu$. 

We refer to the bundle $E^s$ as the \emph{stable bundle} and the bundle $E^u$ as the \emph{unstable bundle}. 
These bundles are tangent to foliations $W^s$ and $W^u$, and the leaves of these foliations are referred to as the \emph{stable and unstable manifolds}. For $x \in \mathbb{T}^3$, we write $W^{s/u}(x)$ to indicate the corresponding leaf containing $x$, and we note that there are local embedded submanifolds $W^{s/u}_{loc}(x) \subseteq W^{s/u}(x)$ which we call the \emph{local stable and unstable manifolds}. 

Given a curve $\gamma$ on $M$, we say that it is a \emph{$u$-path} (respectively \emph{$s$-path}) if the image of $\gamma$ lies in an unstable (respectively stable) manifold. We then say that a curve is a \emph{$us$-path} if it is a finite concatenation of $u$-paths and $s$-paths. Given a point $x \in \mathbb{T}^3$, we say that a curve $\gamma$ is a \emph{$us$-loop based at $x$} if its start and end point is $x$. Every transitive Anosov diffeomorphism $f : \mathbb{T}^3 \rightarrow \mathbb{T}^3$ is \emph{accessible}, meaning that any two points can be connected by a $us$-path.

Recall that the Anosov closing lemma says if an orbit segment $\{f^k(x) \mid 0 \leq k \leq n\}$ satisfies the property that $f^n(x)$ is close to $x$, then one is able to find a point $p \in \mathbb{T}^3$ such that $f^n(p) = p$ and the orbit of $p$ shadows the orbit segment \cite[Corollary 6.4.17]{katok1995introduction}. We actually need a slight improvement on this lemma, which follows immediately from the usual closing lemma and local product structure \cite[Theorem 6.4.15 and Proposition 4.1.16]{katok1995introduction}. 

\begin{proposition} \label{propn:anosov-closing}
    For each integer $m \geq 1$, there exists $C, \varepsilon > 0$ such that for any sequence of points $x_1, \ldots, x_m \in \mathbb{T}^3$ with $d(f^n(x_j), x_j) \leq \varepsilon$ for $1 \leq j \leq m$ and $d(x_j, x_\ell) \leq \varepsilon$ for $1 \leq j, \ell \leq m$, there exists a $p \in \mathbb{T}^3$ satisfying $f^{mn}(p) = p$, and for each $1 \leq j \leq m$ and $0 \leq k \leq n-1$,
    \[d(f^{k + (j-1)n}(p), f^k(x_j)) \leq C \alpha^{\min\{k, n-k\}} \max_{1 \leq j \leq m} d(f^n(x_j), x_j),\]
    where $\alpha \coloneq \max_{x \in \mathbb{T}^3} \{\nu(x), \mu(x)^{-1} \}.$
\end{proposition}

We also recall two important invariant measures for Anosov diffeomorphisms. Let $\mathcal{M}_f$ be the collection of all $f$-invariant Borel probability measures. Given a H\"older continuous function ${\varphi : \mathbb{T}^3 \rightarrow \mathbb{R}}$ we define the \emph{topological pressure} of $\varphi$ to be 
\[ \mathcal{P}(\varphi) \coloneq \sup_{\mu \in \mathcal{M}_f} \left\{ h_f(\mu) +  \int_{\mathbb{T}^3} \varphi \,d\mu\right\},  \]
where $h_f(\mu)$ denotes the metric entropy. A measure achieving the supremum is called an \emph{equilibrium state} for $\varphi$, and it is well-known that for Anosov diffeomorphisms, every H\"older function has a unique associated equilibrium state \cite[Theorem 4.1]{bowen1975equilibrium}. Furthermore, if $\mu$ is an equilibrium state for the functions $\varphi$ and $\psi$, then there must be a H\"older continuous function $u : \mathbb{T}^3 \rightarrow \mathbb{R}$ such that for all $x \in \mathbb{T}^3$, $\varphi(x) = u(f(x)) - u(x) + \psi(x)$ \cite[Proposition 4.5]{bowen1975equilibrium}.  The \emph{measure of maximal entropy} is the equilibrium state associated to the zero function, and the \emph{SRB measure} is the equilibrium state associated to the function $x \mapsto \det(D^u_xf)$, where $D^uf$ is the \emph{unstable derivative} $D^uf \coloneq Df|_{E^u}$.

Finally, we make the standing assumption throughout the paper that $\dim(E^u) = 2$ and ${\dim(E^s) = 1}$. By considering the reverse dynamical systems if needed, it is sufficient to work with such an assumption.

\subsection{Regularity of foliations}
We now discuss some facts about the regularity of the stable and unstable manifolds (see also \cite[Sections 2.1 and 4]{GRH}). Note that for an Anosov diffeomorphism $f$, there exists constants $1 < \mu_- < \mu_+$, $1 < \nu_- < \nu_+$, and $C>0$ such that for every $n \geq 0$ and $v^{s/u} \in E^{s/u}$, we have
\begin{equation*}
   C^{-1}\mu_+^{-n}||v^{s}|| \leq ||Df^{n}(v_s)|| \leq C\mu_-^{-n} ||v^{s}|| \quad \text{and}
\quad 
     C^{-1}\nu_-^{n}||v^{u}|| \leq ||Df^{n}(v_u)|| \leq C\nu_+^{n} ||v^{u}||.
\end{equation*}
The \emph{stable and unstable bunching parameters} are given by 
\[ b^s(f) \coloneq \frac{\log(\mu_-) + \log(\nu_-)}{\log(\nu_+)}\quad \text{and}\quad b^u(f) \coloneq \frac{\log(\mu_-) + \log(\nu_-)}{\log(\mu_+)}.\]
 It was shown by Hasselblatt in \cite{hasselblatt1997regularity} that for $\varepsilon > 0$ sufficiently small the stable foliation and distribution are $C^{b^s(f) - \varepsilon}$ regular. Similarly, we have that the unstable foliation and distribution are $C^{b^u(f) - \varepsilon}$ regular for $\varepsilon$ sufficiently small. Finally, if we assume that $f$ is a small perturbation of an Anosov automorphism $L : \mathbb{T}^3 \rightarrow \mathbb{T}^3$ which has one dimensional stable subbundle and a pair of complex conjugate eigenvalues, then we may assume that the stable foliation and bundle are $C^{3- \varepsilon}$ and the unstable foliation and bundle are $C^{(3/2)-\varepsilon}$, where $\varepsilon > 0$ is sufficiently small depending on the perturbation.

\subsection{Normal forms}
Consider a hyperbolic automorphism ${L: \T^{3}\to \T^{3}}$ with a pair of complex conjugate eigenvalues outside the unit disk. The following result is an adapted version of the existence of normal forms for a small perturbation of $L$ (see also \cite{guysinsky1998normal}).
\begin{proposition}\label{propn:normalforms}
There exists a $C^{1}$-neighborhood $\mathcal{U}$ of $L$
such that for every $C^{r}$-Anosov diffeomorphism $f \in \mathcal{U}$, there exists a family of maps $\mathcal{H}_{x}: E^{u}(x) \to W^{u}(x)$ such that 
\begin{enumerate}
    \item $\mathcal{H}_{x}$ is a $C^{r}$ diffeomorphism for all $x \in M$,
    \item $\mathcal{H}_{x}(0)=x$,
    \item $D_{0}\mathcal{H}_{x}(0)=\text{Id}$,
    \item $D \mathcal{H}_{x}$ is Lipschitz along $W^u(x)$,
    \item if $y \in E^{u}(x)$, then $\mathcal{H}_{y}^{-1} \circ \mathcal{H}_{x}: E^{u}(x) \to E^{u}(y)$ is affine, and
    \item the map $x \to \mathcal{H}_{x}$ from $M$ to $\textrm{Imm}^{r} (E^{u}(x),\T^{3})$ is H\"older.
\end{enumerate}\end{proposition}
\begin{proof}[Sketch of Proof]
    Using the same proof as \cite[Proposition 4.1]{sadovskaya2005uniformly}, the regularity degree of the maps $\mathcal{H}_{x}$ is greather than the quasiconformality coefficient
    $$
    \frac{\log(\max\{\|D_{x}^{u}f\|^{-1}:x \in \T^{3}\})}{\log(\max \{\|D_x^{u}f\|: x \in \T^{3}\})},
    $$
    which can made arbitrarily close to 1 in a $C^{1}$-neighborhood of $L$.
\end{proof}

\subsection{Holonomies and linear cocycles}

We follow \cite{kalinin2013cocycles,sadovskaya2015cohomology}. For $\beta \in (0,1]$, let $p : \mathcal{E} \rightarrow \mathbb{T}^3$ be a finite dimensional $\beta$-H\"older vector bundle over $\mathbb{T}^3$. Given $x \in \mathbb{T}^3$, we denote the fiber over $x$ by $\mathcal{E}_x \coloneq p^{-1}(\{x\})$. Recall that there is a locally trivial open cover $\{U_i\}_{i=1}^k$ of $\mathbb{T}^3$ with homeomorphisms 
\[ \phi_i : p^{-1}(U_i) \rightarrow U_i \times \mathbb{R}^d, \quad \phi_i(v) \coloneq (p(v), \Phi_i(v))\]
such that $\phi_j \circ \phi_i^{-1}$ is a homeomorphism and the transition map $\tau_{ij}$ defined by the relationship $\varphi_j \circ \varphi_i^{-1}(x,v) = (x, \tau_{ij}(x)v)$ depends $\beta$-H\"older on $x$. Taking a partition of unity $\{\rho_i\}$ subordinate to our locally trivial cover $\{U_i\}$, we can identify $\mathcal{E}$ as a $\beta$-H\"older subbundle of the trivial bundle via 
\[ \psi : \mathcal{E} \rightarrow \mathbb{T}^3 \times \mathbb{R}^{kd}, \quad \psi(v) \coloneq (p(v), \rho_1 \Phi_1(v) \times \cdots \times \rho_k \Phi_k(v)).\] 
In particular, there is a natural way of equipping $\mathcal{E}$ with a Riemannian metric $\|\cdot\|$ through a partition of unity.

Let $F : \mathcal{E} \rightarrow \mathcal{E}$ be a bundle automorphism. We say that $F$ is a \emph{linear cocycle} over $f$ if $p \circ F = f \circ p$. Since $F$ is a bundle automorphism, notice that for each $x \in \mathbb{T}^3$, the corresponding map between fibers $F_x : \mathcal{E}_x \rightarrow \mathcal{E}_{f(x)}$ is an isomorphism. For nearby points $x,y \in \mathbb{T}^3$, let $I_{x,y} : \mathcal{E}_x \rightarrow \mathcal{E}_y$ be an identification of the fibers. We say that a linear cocycle $F$ over $f$ is \emph{$\eta$-H\"older} for some $\eta \in (0,\beta]$ if for all nearby points $x,y \in \mathbb{T}^3$, we have
\[ \|F_x - I_{f(x), f(y)}^{-1} \circ F_y \circ I_{x,y}\| \leq C d(x,y)^\eta.\]
Given an $\eta$-H\"older cocycle $F$ over $f$ and given $x \in \mathbb{T}^3$ and $n \in \mathbb{Z}$, we write 
\[ F_x^n \coloneq \begin{dcases} F_{f^{n-1}(x)} \circ \cdots \circ F_x & \text{if } n \geq 1, \\ \text{Id} & \text{if } n = 0, \\ (F^{-n}_{f^n(x)})^{-1} &\text{if } n \leq -1.  \end{dcases}\]
We say that $F$ is \emph{fiber bunched} if there exists $\theta \in (0,1)$ and $L > 0$ such that for all $x \in \mathbb{T}^3$ and $n \geq 0$, we have
\[  \|F_x^n\| \|(F_x^n)^{-1}\| \nu(x,n)^{\eta} < L \theta^n \quad\text{and}\quad \|F_x^{-n} \| \|(F_x^{-n})^{-1}\| \mu(x,-n)^{-\eta} < L \theta^n.\]
If $F$ is a fiber bunched cocyle over $f$, then one is able to show that there exists a unique family of holonomies in the following sense.

\begin{proposition}[{\cite[Proposition 4.4]{sadovskaya2015cohomology}}] \label{propn:holonomies}
    If $F$ is a $\eta$-H\"older continuous fiber bunched cocycle over $f$, then for each $x \in \mathbb{T}^3$ and $y \in W^{s/u}_{loc}(x)$, there exists a unique linear map $H^{s/u}_{x,y}(F) : \mathcal{E}_x \rightarrow \mathcal{E}_y$ which satisfies the following.
    \begin{enumerate}[(a)]
    \item We have $H^{s/u}_{x,x}(F) = \text{Id}$ and $H^{s/u}_{y,z}(F) \circ H^{s/u}_{x,y}(F) = H^{s/u}_{x,z}(F)$.
    \item For all $n \in \mathbb{N}$, we have that
    \begin{itemize}
        \item if $y \in W^s_{loc}(x)$, then $H^s_{x,y}(A) = (F^n_y)^{-1} \circ H^s_{f^n(x), f^n(y)}(F) \circ F^n_{x}$, and 
        \item if $y \in W^u_{loc}(x)$, then 
        $H^u_{x,y}(F) = F^{-n}_y \circ H^u_{f^{-n}(x), f^{-n}(y)}(F)  \circ (F^{-n}_{x})^{-1}.$
    \end{itemize}
    \item We have $\|H^{s/u}_{x,y}(F) - \text{Id}\| \leq c d(x,y)^\eta$ for some $c > 0$ independent of $x$ and $y$.
    \end{enumerate} 
    Moreover, these can be extended to the global stable and unstable manifolds, and for every $x \in \mathbb{T}^3$ we have
    \[ \begin{cases} H^s_{x,y}(F) = \lim_{n \rightarrow \infty} (F^n_y)^{-1}  \circ F^n_{x} &\text{ for } y \in W^s(x),\\ 
    H^u_{x,y}(F) = \lim_{n \rightarrow \infty} F^n_{f^{-n}(y)} \circ (F^n_{f^{-n}(x)})^{-1} &\text{ for } y \in W^u(x). \end{cases}\]    
\end{proposition}

Following \cite{katok1996cocycles}, given a $u$-path $\gamma$ we define the \emph{periodic cycle holonomy} by
\[ PCH_\gamma(F) := H^u_{\gamma(0), \gamma(1)}(F).\]
We can do the same with $s$-paths, and we can then extend this to $us$-paths by setting concatenations as products. We observe some simple consequences of the definition and Proposition \ref{propn:holonomies}.

\begin{proposition} \label{propn:pcf}
    If $F$ is a $\gamma$-H\"older continuous fiber bunched cocycle over $f$, then the periodic cycle holonomy satisfies the following.
    \begin{enumerate}[(a)]
        \item If $\gamma_1$ and $\gamma_2$ are two $us$-curves, then 
        \[ PCH_{\gamma_2 * \gamma_1}(F) = PCH_{\gamma_1}(F) \circ PCH_{\gamma_2}(F).\]
        \item If $\gamma$ is a $us$-curve connecting $x$ and $y$, then 
        \[ PCH_{f(\gamma)}(F) = (F_y)^{-1} \circ PCH_\gamma(F) \circ F_x.\]
    \end{enumerate}
\end{proposition}

\section{Periodic cycle holonomies and regularity} \label{sec:pcf}

In this section, we describe more properties of the unstable derivative cocycle $D^uf$ over $f$, and we also define the quadrilateral holonomy map. 

Let $L : \mathbb{T}^3 \rightarrow \mathbb{T}^3$ be a hyperbolic toral automorphism whose spectrum contains a pair of non-real complex conjugate eigenvalues, and let $f : \mathbb{T}^3 \rightarrow \mathbb{T}^3$ be a transitive $C^r$-Anosov diffeomorphism close to $L$ which satisfies our standing assumption that the unstable has dimension two. Our first proposition shows that as long as $f$ is close to $L$, the unstable derivative cocycle is fiber bunched.

\begin{proposition} \label{propn:fiberbunched}
	For every $0 < \theta < 1$, there exists a $C^{1}$-neighborhood $\mathcal{U}_{\theta}$ of $L$ and a constant $C > 0$ such that for every $C^{r}$-Anosov diffeomorphism $f \in \mathcal{U}_{\theta}$ and $a,b \in \T^{3}$, we have
  \[ 
    \|D_{a}^{u}f^{-n}\|  \|D_{b}^{u}f^{n}\| \leq C\theta^{-n} ,\quad \text{for all } n\geq 1
    \]
	In particular, the unstable derivative cocycle of every $f \in \mathcal{U}_{\theta}$ is fiber bunched for appropriate choice of $\theta$.
\end{proposition}

\begin{proof}
Let $E^u$ be the unstable bundle of $L$, let $L_u \coloneq L|_{E^u}$, and pass to a Riemannian metric which is adapted to $E^u$ in such a way so that $\|L_u^{-1}\| = \|L_u\|^{-1}$.
    Observe that for all $\varepsilon > 0$, we can find a $C^1$-neighborhood $\mathcal{U}_{\theta, \varepsilon}$ such that for all $f \in \mathcal{U}_{\theta, \varepsilon}$,
	\[\begin{split} \|D^u_a f^{-n}\| \|D^u_bf^n\| \theta^n &\leq (\|L_u\| + \varepsilon)^n(\|L_u\|^{-1} + \varepsilon)^n \theta^n \\
		& \leq (1 + (\|L_u\| + \|L_u\|^{-1})\varepsilon + \varepsilon^2)^n \theta^n. \end{split}\]
		Furthermore, note that for $\varepsilon$ sufficiently small, we have 
		\[ (1 + (\|L_u\| + \|L_u\|^{-1})\varepsilon + \varepsilon^2) \theta < 1.\]
		The first claim follows for appropriately small $\varepsilon$ and by using the fact that all norms are equivalent. By choosing appropriate $\theta$ and $\varepsilon$, it follows from the definition that the unstable derivative cocycle is fiber bunched.
\end{proof}

As a consequence of Proposition \ref{propn:holonomies}, we deduce that one can choose a $C^1$-neighborhood $\mathcal{U}$ of $L$ such that the stable and unstable holonomies are well-defined for all $C^r$-Anosov diffeomorphisms $f \in \mathcal{U}$. We are now interested in the regularity of these holonomies. 
The following is a well-known observation by Shub, and we include a short proof for the reader's convenience.
\begin{proposition}\label{propn:normalforms2}
    There exists a $C^{1}$-neighborhood $\mathcal{U}$ of $L$ such that for every $C^{r}$-Anosov diffeomorphism $f \in \mathcal{U}$ and $a \in \T^{3}$, the unstable holonomy map $W^u(a) \ni x \mapsto H^u_{a,x}(D^uf)$ is $C^{r-1}$-differentiable.

\end{proposition}
\begin{proof}
Let $\{\mathcal{H}_{x}: E^{u}(x) \to W^{u}(x)\}_{x \in \T^{3}}$ be the family of normal forms given in Proposition \ref{propn:normalforms}. We claim that for every $x,y \in \T^{3}$ with $y \in W^{u}(x)$ we can write $H^{u}_{x,y}(D^{u}f)=D_{0}(\mathcal{H}_{y}^{-1} \circ \mathcal{H}_{x})$. We will do so by checking that the derivative of the normal forms acts as a holonomy, and then use uniqueness of holonomies to deduce the result.

First, notice that the chain rule yields that for every $y,z \in W^u(x)$, we have
\[ D_{0}(\mathcal{H}_{z}^{-1}\circ \mathcal{H}_x)=D_{0}(\mathcal{H}_{z}^{-1}\circ \mathcal{H}_y\circ \mathcal{H}_{y}^{-1}\circ \mathcal{H}_x)=D_{0}
(\mathcal{H}_{z}^{-1}\circ \mathcal{H}_y)D_{0}(\mathcal{H}_{y}^{-1}\circ \mathcal{H}_x) .  \]
We also observe that for $y \in W^u(x)$ and $n \geq 0$, we can use Proposition \ref{propn:normalforms} to get
\[
    \begin{split}
 D_{0}(\mathcal{H}_{y}^{-1} \circ \mathcal{H}_{x}) &=  D_{0}((f^{-n}\circ \mathcal{H}_{y}  )^{-1}\circ (f^{-n} \circ \mathcal{H}_{x})) \\
 & =D_{0}((\mathcal{H}_{f^{n}(y)}\circ D^{u}_{y}f^{n})^{-1} \circ (\mathcal{H}_{f^{n}(x)} \circ D^{u}_{x}f^{n})) \\
 &= D_{0}((D^{u}_{y}f^{n})) ^{-1} \circ \mathcal{H}_{f^{n}(y)}^{-1} \circ \mathcal{H}_{f^{n}(x)} \circ D_x^{u}f^{n})\\
 &= (D^{u}_{y}f^{n}) ^{-1} D_0(\mathcal{H}_{f^{n}(y)}^{-1} \circ \mathcal{H}_{f^{n}(x)}) D^{u}_xf^{n}.
 \end{split}
\]
Finally, one more application of Proposition \ref{propn:normalforms} shows that there exists a uniform $c > 0$ such that for $x, y \in \mathbb{T}^3$ sufficiently close, we have
\[ \| D_{0}(\mathcal{H}_{y}^{-1}\circ \mathcal{H}_{x}) - \text{Id} \| \leq c d(x,y).\]
Therefore, by the uniqueness of holonomies in Proposition \ref{propn:holonomies}, we have $H^{u}_{x,y}(D^{u}f)=D_{0}(\mathcal{H}_{y}^{-1} \circ \mathcal{H}_{x})$, and Proposition \ref{propn:normalforms} shows that the map $x \mapsto H^{u}_{a \to x}(D^{u}(f))$ is $C^{r-1}$-differentiable.
\end{proof}

We define the \emph{stable holonomy} of $D^uf$ by
\[
W^u(a) \ni x \mapsto H^{s}_{x,\text{Hol}^s_{a,b}(x)}(D^{u}f).
\]
The aim is to show that this map is $C^1$ provided $f$ is $C^1$-close enough to $L$.
\begin{proposition}\label{prop:stableholonomyregularity}
    There exists a $C^{1}$-neighborhood of $L$ such that the stable holonomy map given by $x \mapsto H^{s}_{x,\text{Hol}^s_{a,b}(x)}(D^{u}f)$ is $C^{1}$.
\end{proposition}

\begin{proof}
It is enough to show that the map is regular on $W^{u}_{loc}(a)$. Notice that for every $n \geq 0$, we have $\text{Hol}^s_{a,b}=f^{-n} \circ \text{Hol}^s_{f^{n}(a),f^{n}(b)} \circ f^{n}$, which by the chain rule implies 
\[
D_{x}\text{Hol}^s_{a,b}=(D_{\text{Hol}^s_{a,b}(x)}^{u}f^{n})^{-1} \circ D_{f^{n}(x)}  \text{Hol}^s_{f^{n}(a),f^{n}(b)} \circ D^{u}_{x}f^{n}.
\]
Now if $I_{f^{n}(x),\text{Hol}^s_{f^{n}(a),f^{n}(b)}(f^{n}(x))}$ are the identifications along the sequence of points $f^{n}(x)$ and $\text{Hol}^s_{f^{n}(a), f^{n}(b)}(f^{n}(x))$ for $n$ large, we have 
\begin{equation*}
\begin{split}
\|D_{x}\text{Hol}^s_{a,b}-&(D^u_{\text{Hol}^s_{a,b}(x)}f^{n})^{-1} \circ I_{f^{n}(x), \text{Hol}^s_{f^n(a), f^n(b)}(f^n(x))} \circ D^{u}_{x}f^{n}\| \\
&\leq \|(D^{u}_{\text{Hol}^s_{a,b}(x)}f^{n})^{-1}\|\cdot \|D_{f^{n}(x)}\text{Hol}^s_{f^{n}(a),f^{n}(b)}-I_{f^{n}(x), \text{Hol}^s_{f^n(a), f^n(b)}(f^n(x)) }\|\cdot  \|D^{u}_{x}f^{n}\|.
\end{split}
\end{equation*}
Fix $0 < \theta < 1$ so that $\sup_{x\in \mathbb{T}^{3}} \|D_{x}f|_{E^{s}(x)}\| < \theta$. Recall that by taking the $C^1$ neighborhood small, the stable holonomy is uniformly $C^2$ and the identifications are $\beta$-H\"older continuous for some $0 < \beta < 1$. Consequently, using the fact that $a \in W^s(b)$, there is a $C>0$ such that for all $n \geq 0$ 
\[
\|D_{f^{n}(x)}\text{Hol}^s_{f^{n}(a),f^{n}(b)}-I_{f^n(x), \text{Hol}^s_{f^n(a), f^n(b)}(f^n(x))}\| \leq C\theta^{\beta n}.
\]

Furthermore, adjusting $C$ if necessary, we can use Proposition \ref{propn:fiberbunched} to guarantee the following for every $n$ large:
\[
\|(D^{u}_{\text{Hol}^s_{a,b}(x)}f^{n})^{-1}\|\cdot \|D^{u}_{x}f^{n}\| \leq C \theta^{-\frac{n\beta }{2}}.
\]
Combining the above, we deduce that 
\[ \|D_{x}\text{Hol}^s_{a,b}-(D^u_{\text{Hol}^s_{a,b}(x)}f^{n})^{-1} \circ I_{f^{n}(x), \text{Hol}^s_{f^n(a), f^n(b)}(f^n(x))} \circ D^{u}_{x}f^{n}\| \leq C^2\theta^{\frac{n\beta }{2}},\]
which goes to $0$ as $n \to \infty$. We conclude that $D_{x}\text{Hol}^s_{a,b} = H^{s}_{x, \text{Hol}^s_{a,b}} (D^{u}f)$, implying the statement.
\end{proof}

Now that we have shown that the holonomies have a certain degree of regularity, we will define the smooth structures that are going to be matched under the conjugate dynamics. For a pair of points $a,b \in \T^{3}$ in the same stable leaf, we define the \emph{quadrilateral holonomy map} 
\[ \Phi_{a,b}: W^u_{loc}(a) \rightarrow \text{GL}(E^u(a)), \quad \Phi_{a,b}(x) \coloneq H_{x,a}^{u}(D^{u}f) \, H_{\text{Hol}^s_{a,b}(x),x}^{s}(D^{u}f)\, H_{b, \text{Hol}^s_{a,b}(x)}^{u}(D^{u}f)\, H_{a,b}^{s}(D^{u}f) 
\] 
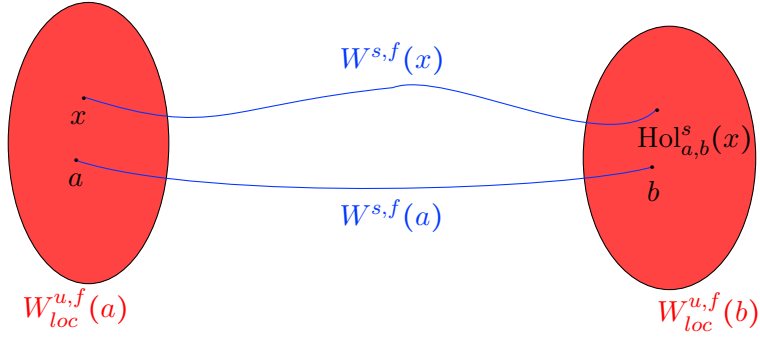
\begin{figure}[H]
\begin{center}
    \begin{tikzpicture}[x=0.75pt,y=0.75pt,yscale=-1,xscale=1]

\draw  [fill={rgb, 255:red, 255; green, 14; blue, 14 }  ,fill opacity=0.78 ] (404,144.56) .. controls (404,107.84) and (423.48,78.08) .. (447.5,78.08) .. controls (471.52,78.08) and (491,107.84) .. (491,144.56) .. controls (491,181.27) and (471.52,211.04) .. (447.5,211.04) .. controls (423.48,211.04) and (404,181.27) .. (404,144.56) -- cycle ;
\draw  [fill={rgb, 255:red, 255; green, 14; blue, 14 }  ,fill opacity=0.78 ] (114,137.06) .. controls (114,97.86) and (132.13,66.08) .. (154.5,66.08) .. controls (176.87,66.08) and (195,97.86) .. (195,137.06) .. controls (195,176.26) and (176.87,208.04) .. (154.5,208.04) .. controls (132.13,208.04) and (114,176.26) .. (114,137.06) -- cycle ;
\draw [color={rgb, 255:red, 0; green, 51; blue, 255 }  ,draw opacity=1 ][fill={rgb, 255:red, 0; green, 0; blue, 0 }  ,fill opacity=0 ]   (149,146.08) .. controls (215,167.08) and (397,161.08) .. (439,149.08) ;
\draw [color={rgb, 255:red, 0; green, 51; blue, 255 }  ,draw opacity=1 ][fill={rgb, 255:red, 0; green, 0; blue, 0 }  ,fill opacity=0 ]   (153,114.08) .. controls (219,135.08) and (222,118.08) .. (308,109.08) .. controls (335,99.08) and (416,145.08) .. (441,120.08) ;

\draw (143,151) node [anchor=north west][inner sep=0.75pt]    {$a$};
\draw (435,155) node [anchor=north west][inner sep=0.75pt]    {$b$};
\draw (144.5,143.5) node [anchor=north west][inner sep=0.75pt]  [font=\Large]  {$.$};
\draw (435,147) node [anchor=north west][inner sep=0.75pt]  [font=\Large]  {$.$};
\draw (280,164) node [anchor=north west][inner sep=0.75pt]  [color={rgb, 255:red, 1; green, 56; blue, 255 }  ,opacity=1 ]  {$W^{s,f}( a)$};
\draw (120,209) node [anchor=north west][inner sep=0.75pt]  [color={rgb, 255:red, 255; green, 0; blue, 0 }  ,opacity=1 ]  {$W_{loc}^{u,f}( a)$};
\draw (440,211) node [anchor=north west][inner sep=0.75pt]  [color={rgb, 255:red, 255; green, 0; blue, 0 }  ,opacity=1 ]  {$W_{loc}^{u,f}( b)$};
\draw (144,120) node [anchor=north west][inner sep=0.75pt]    {$x$};
\draw (148.4,112) node [anchor=north west][inner sep=0.75pt]  [font=\Large]  {$.$};
\draw (430,127) node [anchor=north west][inner sep=0.75pt]    {$\text{Hol}^{s}_{a,b}(x)$};
\draw (437.5,118) node [anchor=north west][inner sep=0.75pt]  [font=\Large]  {$.$};
\draw (280,85) node [anchor=north west][inner sep=0.75pt]  [color={rgb, 255:red, 1; green, 56; blue, 255 }  ,opacity=1 ]  {$W^{s,f}( x)$};

\end{tikzpicture}
\caption{An illustration of the quadrilateral}
\end{center}
\end{figure}

Notice that the above map can be extended uniquely to a continuous map to the global unstable manifold $W^{u}(a)$ using an extension of the stable holonomy map $\text{Hol}^s_{a,b}$. The definition of this map is inspired by the work of Katok and Konenko along with the work of Gogolev and Rodriguez Hertz \cite{katok1996cocycles, GRH}. The following is a consequence of the previous study on the regularity of holonomies.

\begin{proposition} \label{propn:regholonomy}
For every $a,b \in \mathbb{T}^{3}$ in the same stable leaf, the quadrilateral holonomy map $\Phi_{a,b}$ is $C^{1}$.
\end{proposition}

\begin{proof}
Apply Propositions \ref{propn:normalforms} and \ref{prop:stableholonomyregularity} in order to obtain the $C^{1}$ regularity of $\Phi_{a,b}$ along the unstable leaf of $a$.
\end{proof}

\section{Proof of Theorem \ref{thm:main}} \label{sec:proof}

From now on, we will be working with pairs of Anosov diffeomorphisms $f,g : \mathbb{T}^3 \rightarrow \mathbb{T}^3$ satisfying the standing assumption that the unstable has dimension two. To indicate the correct object for the corresponding diffeomorphism, we will use a superscript with the corresponding letter, e.g., the unstable bundle for $f$ will be denote by $E^{u,f}$.
The aim of this section is to prove Theorem \ref{thm:main} assuming two things. First, we assume the following ``matching theorem,'' which will proven in Section \ref{sec:parry}.

\begin{theorem}[Matching theorem] \label{thm:matchingtheorem}
    Let $L: \mathbb{T}^{3} \to \mathbb{T}^{3}$ be an Anosov toral automorphism with a pair of non-real complex conjugate eigenvalues. There exists a $C^{1}$-neighborhood $\mathcal{U}$ of $L$ such that if $f,g \in \mathcal{U}$ are a pair of smooth Anosov diffeomorphisms with non-constant matching periodic data and are such that the SRB and measure of maximal entropy for $f$ coincide, then there exists a H\"older bundle map ${C: E^{u,f} \to E^{u,g}}$ such that on the fibers we have $C_x : E^{u,f}_x \rightarrow E^{u,g}_{h(x)}$ and $C$ conjugates the unstable cocycles in the following sense:
    \[ D_x^{u}f=(C_{f(x)})^{-1}\, (D_{h(x)}^{u}g) \, C_{x}.\]
\end{theorem}

We will additionally assume that there exists $a,b \in \mathbb{T}^3$ in the same stable leaf such that the quadrilateral holonomy map $\Phi^f_{a,b}$ is non-constant. Otherwise, every periodic cycle holonomy of $D^uf$ will be identity, which, as we will show in Proposition \ref{propn:KK}, implies the pair has constant matching periodic data. With both of these, we have the tools necessary to prove Theorem \ref{thm:main}.

First, notice that Theorem \ref{thm:matchingtheorem} implies that if $f,g$ are $C^r$-Anosov diffeomorphisms in a sufficiently small $C^{1}$ neighborhood of $L$ with non-constant matching periodic data, then for every $a,b \in \T^{3}$ in the same stable leaf of $f$, the periodic holonomy cocyles match in the following sense:
    \[
    \Phi^f_{a,b}= (C_a)^{-1}\, \Phi^{g}_{h(a),h(b)}\, C_a. 
    \]
This constitutes a strong form of matching, since for $f$ and $g$ in a $C^1$ neighborhood of $L$, the quadrilateral holonomy maps are well-defined and are $C^1$ due to Proposition 
\ref{propn:regholonomy}. We now mimic the scheme of \cite{GRH}. 
In particular, since there is an $a,b \in \mathbb{T}^3$ which lie in the same stable leaf such that $\Phi^f_{a,b}$ is non-constant, there exists a neighborhood $V \subset W_{loc}^{u,f}(a)$ such that $D_x\Phi^f_{a,b} \neq 0$ for every $x \in V$. 

Let $p \in \T^{3}$ be a fixed point for $f$ such that $D_x^{u}f$ has a pair of complex non-real conjugate eigenvalues. Notice that such point always exists when we consider diffeomorphisms in a $C^{1}$ neighborhood of $L$. Due to the minimality of the stable foliation $W^{s,f}$, there exists $z \in V$ such that $p \in W^{s,f}(z)$. Let $c=\text{Hol}^s_{p,z}(a) \in W^{u}(a)$.
We now adapt the arguments in the commutative setting to the non-commutative setting. We say that a quadruple of points $x_{1}, x_{2}, x_{3}, x_{4} \in \T^{3}$ forms a \emph{$us$-quadrilateral} for $f$ if $x_{2} \in W^{s}(x_{1}) \cap W^{u}(x_{3}) $ and $x_{4} \in W^{u}(x_{1}) \cap W^{s}(x_{3})$. For $f$ $C^1$-close to $L$ and a $us$-quadrilateral $(x_1,x_2,x_3,x_4)$ for $f$ we write 
$$
H_f(x_1,x_2,x_3,x_4)=H_{x_4,x_1}^{u}(D^{u}f) \, H_{x_3,x_4}^{s}(D^{u}f) \, H_{x_2,x_3}^{u}(D^{u}f) \, H_{x_1,x_2}^{s}(D^{u}f).
$$
Note that we are only introducing $H_f$ for notational ease -- in application, we will just be using $H_f$ for the ordered points $a$, $c$, $b$, $\text{Hol}^s_{a,b}(x)$, $(\text{Hol}^s_{p,z})^{-1}(x)$, which gives rise to $\Phi^f_{a,b}$.

The following will be useful.
\begin{lemma} \label{lem:quadhol}
    If $x_{1},x_{2},x_{3},x_{4},x_{5},x_{6} \in \T^{3}$ are six points such that the quadruples $(x_{1},x_{2},x_{5},x_{6})$ and $(x_{2},x_{3},x_{4},x_{5})$ form $us$-quadrilaterals for $f$, then
    \[
H_f(x_{1},x_{3},x_{4},x_{6})=(H^{s}_{x_{1},x_{2}}(D^u f ))^{-1} \, H_f(x_{2},x_{1},x_{6},x_{5})^{-1} \, H_f(x_{2},x_{3},x_{4},x_{5}) \, H^{s}_{x_1,x_2}(D^uf)
    \]
\end{lemma}

\begin{proof}
The product of holonomies on the $us$-quadrilaterals $(x_1,x_2,x_5,x_6)$ and $(x_2,x_3,x_4,x_5)$ satisfies the following: 
\begin{equation}\label{lemmaquadholeq}
    \begin{split}  
&H_f(x_2,x_3,x_4,x_5)=H_{x_5,x_2}^{u}(D^{u}f)H_{x_4,x_5}^{s}(D^{u}f)H_{x_3,x_4}^{u}(D^{u}f)H_{x_2,x_3}^{s}(D^{u}f), \quad \text{and} \\
&H_f(x_2,x_1,x_6,x_5)^{-1}=H_{x_1,x_2}^{s}(D^{u}f)H_{x_6,x_1}^{u}(D^{u}f)H_{x_5,x_6}^{s}(D^{u}f)H_{x_2,x_5}^{u}(D^{u}f).
\end{split}
\end{equation}
Taking into consideration the properties of holonomies in Proposition \ref{propn:holonomies}, we take the product of the formulae in Equation \eqref{lemmaquadholeq} to get
\[\begin{split}
H_f(x_{2},x_{1},x_{6},x_{5})^{-1}&H_f(x_{2},x_{3},x_{4},x_{5}) \\&=H_{x_1,x_2}^{s}(D^{u}f)H_{x_6,x_1}^{u}(D^{u}f)H_{x_5,x_6}^{s}(D^{u}f)H_{x_4,x_5}^{s}(D^{u}f)H_{x_3,x_4}^{u}(D^{u}f)H_{x_2,x_3}^{s}(D^{u}f)\\
&=H^{s}_{x_1,x_2}(D^{u}f)H^{u}_{x_{6},x_1}H_{x_4,x_6}^{s}(D^{u}f)H^{u}_{x_3,x_4}(D^{u}f)H^{s}_{x_1,x_3}(D^{u}f)H^{s}_{x_{1},x_{2}}(D^{u}f)^{-1}\\
&=H^{s}_{x_{1},x_{2}}(D^{u}f)H_{f}(x_1,x_3,x_4,x_6)H^{s}_{x_1,x_2}(D^{u}f)^{-1},
\end{split}
\]
which is equivalent to what we want to prove.
\end{proof}

Fixing arbitrary $x \in V$, we can apply Lemma \ref{lem:quadhol} to get
\[
\Phi_{a,b}^f(\text{Hol}^s_{p,z}(x))=H_{a,c}^{s}(D^{u}f)^{-1}\,(\Phi_{c,a}^f(x))^{-1}\, \Phi^f_{c,b}(x) \,H^{s}_{a,c}(D^{u}f).
\]
Since $D_x\Phi^f_{a,b}$ is non-zero for every $x \in V$, the product rule yields that for every $x \in V$, either $D_x(\Phi^f_{c,a})^{-1}$ or $D_x\Phi^f_{c,b}$ is non-zero. In any case, since $p \in V$, we may assume without loss of generality that $D_p\Phi_{c,a}$ is non-zero. Notice that the unstable manifold $W^{u,f}(p)$ maps into itself since $p$ is a fixed point of $f$. Therefore, by Theorem \ref{thm:matchingtheorem}, we have the following matching relations for every $x \in W^{u}(p)$:
\[ 
\begin{split}
    \Phi^{f}_{c,a}(x)=(C_{c})^{-1}\Phi^g_{h(c),h(a)}(h(x))C_{c} \quad \text{and} \quad \Phi^{f}_{c,a}(f|_{W^{u}(p)}(x))=(C_{c})^{-1}\Phi^g_{h(c),h(a)}(f|_{W^{u}(p)}(h(x)))C_{c}.
\end{split}
\]
Using a coordinate system for $E^{u,f}_{a}$, we can view the maps $\Phi^{f}_{c,a}$ and $\Phi^{g}_{h(c),h(a)}$ as $GL(2,\R)$-valued functions, and, using the fact that $D_{p}\Phi^{f}_{c,a} \neq 0$, we get a pair of $C^{1}$-functions $\varphi^{f}_{c,a}: W^{u,f}(a) \to \mathbb{R}$ and $\varphi^{g}_{c,a}: W^{u,g}(h(p)) \to \mathbb{R}$ such that $D_{p}\varphi^{f}_{p}\neq 0 \neq D_{h(p)}\varphi^{g}_{h(p)}$ and which satisfy the matching relations for every $x \in W^u(p)$:
\[ 
    \varphi^{f}_{c,a}(x)=\varphi_{h(c),h(a)}(h(x)) \quad\text{and}\quad    \varphi^{f}_{c,a}(f|_{W^{u}(p)}(x))=\varphi_{h(c),h(a)}(f|_{W^{u}(p)}(h(x))),
\]
Since the differentials of the functions $\varphi^{f}_{c,a}$ and $\varphi^{g}_{h(c),h(a)}$ do not vanish at $p$ and $h(p)$, respectively, we get that both maps 
\[ \Psi^f=(\varphi^{f}_{c,a}, \varphi^{f}_{c,a} \circ f|_{W^{u}(p)}) \quad \text{and}\quad \Psi^{g}=(\varphi^{g}_{c,a}, \varphi^{g}_{c,a} \circ g|_{W^{u}(h(p))})\] 
 are $C^{1}$ local diffeomorphisms around $p$ and $h(p)$ which satisfy $\Psi^{f}=\Psi^{g} \circ h$. Note here that we are using the fact that $L$ has non-real complex conjugate eigenvalues in order to ensure that $D_p\Psi^f$ and $D_{h(p)}\Psi^g$ are invertible. Therefore $h$ is a $C^{1}$ diffeomorphism in a neighborhood $\mathcal{V}_p$ of $p$. One can then carry the smoothness to the entire manifold in the same way as in \cite[Section 3.5]{GRH}. Let $q \in W^{s,f}(p)$ and let $\mathcal{V}_{q}$ be the image under the holonomy map $\text{Hol}^s_{p,q}$ of $\mathcal{V}_p$. We notice that $h|_{\mathcal{U}_{q}}= \text{Hol}^s_{p,q} \circ h|_{\mathcal{V}_p} \circ (\text{Hol}^s_{p,q})^{-1}$ and therefore $h$ is $C^{1}$ in $\mathcal{V}_{q}$ for every $q \in W^{s}(p)$. These last fact combined with the fact that the stable foliation is minimal imply that $h$ is $C^{1}$ along $W^{u,f}$ on the whole $\T^{3}$. By a bootstrap argument \cite[Proposition 3.3]{GRH}, we have that $h$ is $C^{r}$ along $W^{u,f}$. Since $W^{s,f}$ is one-dimensional, $h$ is smooth along $W^{s,f}$ by the same argument in \cite{de1987invariants}. Using Journ\'e's lemma \cite{journelemma}, $h$ is $C^{r_*}$, where 
 \[ r_* = \begin{cases}r &\text{if } r \notin \mathbb{Z} \\ (r-1) + \text{Lip} &\text{if } r \in \mathbb{Z}. \end{cases}\]
 In particular, $h$ is smooth when $f$ and $g$ are smooth.  $\qed$

\section{Parry and Homoclinic Holonomy Groups} \label{sec:parry}

The goal of this section is to prove the matching theorem (Theorem \ref{thm:matchingtheorem}) as well as discuss the standing assumption from the previous section. To start, note that since $f$ and $g$ are sufficiently close to $L$ and $E^{u,L}$ is a trivial bundle, we may assume that $E^{u,f}$ and $E^{u,g}$ are both trivial bundles. As a consequence, we are able to view the unstable derivative cocycles as $\text{GL}(2,\mathbb{R})$ cocycles over $f$, and we work in this level of generality.

\subsection{Parry monoid and representation}
Let $f : \mathbb{T}^3 \rightarrow \mathbb{T}^3$ be an Anosov diffeomorphism with a fixed point $p$, and let $A : \mathbb{T}^3 \times \mathbb{R}^2 \rightarrow \mathbb{T}^3$ be a H\"older continuous cocycle over $f$ which is fiber bunched. In particular, we can view $A$ as a H\"older continuous map ${A : \mathbb{T}^3 \rightarrow \text{GL}(2,\mathbb{R})}$, and for $x \in \mathbb{T}^3$ and $n \in \mathbb{Z}$ we write 
\[ A^n_x \coloneq \begin{dcases} A(f^{n-1}(x)) \cdots A(x) &\text{if } n \geq 1, \\ \text{Id} &\text{if } n = 0, \\ (A^{-n}_{f^n(x)})^{-1} &\text{if } n \leq -1. \end{dcases} \] 
By abuse of notation, we also refer to the map $A : \mathbb{T}^3 \rightarrow \text{GL}(2,\mathbb{R})$ as a H\"older cocycle over $f$. We say that $A$ is an \emph{almost coboundary} if there exists a H\"older continuous map $C : \mathbb{T}^3 \rightarrow \text{GL}(2,\mathbb{R})$ and a matrix $L \in \text{GL}(2,\mathbb{R})$ such that $A(x) = C(f(x)) \, L \, C(x)^{-1}.$

We say that a $us$-loop $\gamma$ is \emph{homoclinic} if it can be written as a concatenation $\gamma = \gamma_2 * \gamma_1$, where $\gamma_1(1) = \gamma_2(0) \in W^{u}(\gamma(0)) \cap W^{s}(\gamma(0))$, $\gamma_1$ is a $u$-path, and $\gamma_2$ is an $s$-path. In particular, if a $us$-loop is homoclinic, then it connects $x$ to a point $z \coloneq \gamma(1)$ homoclinic to $x$. For the purposes of studying algebraic objects associated to these curves, we identify these curves under the equivalence relation given by reparameterization. This will make concatenation an associative operation, and we can study the related monoid of paths.

For $x \in \mathbb{T}^3$, let $\mathcal{H}_{x}$ be the monoid generated by the collection of homoclinic $us$-loops under concatenation. We refer to $\mathcal{H}_{x}$ as the \emph{Parry monoid} based at $x$.\footnote{This is inspired by \emph{Parry's monoid} in \cite{cekic2025holonomy}, which in turn was inspired by \cite{parry1999livvsic}.} Using the periodic cycle holonomy, we get the \emph{Parry representation} of $A$: 
\[ \rho^A_{x} : \mathcal{H}_{x} \rightarrow \text{SL}(2,\mathbb{R}), \quad \rho^A_{x}(\gamma) := PCH_\gamma(A).\]
We say that $\rho^A_{p}$ is \emph{irreducible} if there are no proper subspaces $V \subseteq \mathbb{R}^2$ satisfying $\rho^A_{p}(\gamma)(V) \subseteq V$ for all $\gamma \in \mathcal{H}_{p}$.

Observe that the Parry monoid generates a group $\mathcal{G}_{x}$ which we call the \emph{Parry group}, and the Parry representation naturally extends to the Parry group via
\[ \rho^A_{p}(\gamma^{-1}) = (\rho^A_{p}(\gamma))^{-1}.\]
The \emph{homoclinic holonomy group} based at $x$ is the image of the Parry group under the Parry representation, and we denote this by $G_x \coloneq \rho^A_x(\mathcal{G}_x)$. We say that the homoclinic holonomy group is \emph{irreducible} if the Parry representation is irreducible. Before moving on, we rephrase a result by Sadovskaya in terms of the homoclinic holonomy group.

\begin{lemma}[{\cite[Proposition 4.7]{sadovskaya2015cohomology}}] \label{lem:extension}
        Let $A,B : \mathbb{T}^3 \rightarrow \text{GL}(2,\mathbb{R})$ be H\"older continuous fiber bunched cocycles over $f$ and let $p$ be a fixed point of $f$. If $C_{p} \in \text{GL}(2,\mathbb{R})$ is such that 
        \begin{enumerate}[(a)]
            \item $A_{p} = C_{p} B_{p} C_{p}^{-1}$, and
            \item $\rho^A_{p}(\gamma) = C_{p} \, \rho^B_{p}(\gamma) \, C_{p}^{-1}$ for every $\gamma \in \mathcal{H}_{p}$,
        \end{enumerate}
        then there exists a unique H\"older continuous conjugacy $C : \mathbb{T}^3 \rightarrow \text{GL}(2,\mathbb{R})$ between $A$ and $B$ such that $C(p) = C_{p}$. Moreover, if $A$ and $B$ take values in a closed subgroup of $\text{GL}(2,\mathbb{R})$, then so does $C$.
\end{lemma}

Next, we explain why we are able to reduce our problem to studying $\text{SL}^{\pm}(2,\mathbb{R})$-cocycles when the cocycle has \emph{cohomologically trivial determinant}, meaning that there is a H\"older continuous function $u : \mathbb{T}^3 \rightarrow \mathbb{R}$ such that 
\[ \det(A(x)) = e^{u(f(x)) - u(x)}.\]

\begin{lemma} \label{lem:normalization}
     Let $f : \mathbb{T}^3 \rightarrow \mathbb{T}^3$ be an Anosov diffeomorphism with a fixed point $p$ and consider $A : \mathbb{T}^3 \rightarrow \text{GL}(2,\mathbb{R})$ a H\"older linear cocycle over $f$ which is fiber bunched, has cohomologically trivial determinant, and satisfies the condition that the periodic cycle holonomy of $A$ vanishes on $us$-loops. Let 
    \[ B : \mathbb{T}^3 \rightarrow \text{SL}^{\pm}(2,\mathbb{R}), \quad B(x) \coloneq |\det(A(x))|^{-1/2} A(x).\]
    Then $B$ is fiber bunched and the periodic cycle holonomy of $B$ vanishes on every $us$-loop. Furthermore, $B$ is an almost coboundary if and only if $A$ is an almost coboundary.
\end{lemma}

\begin{proof}
    Let $s(x) \coloneq |\det(A(x))|^{-1/2}$. Because $A$ has cohomologically trivial determinant, we can write
    \[ s(x) = e^{-(u(f(x)) - u(x))/2},\]
    where $u : \mathbb{T}^3 \rightarrow \mathbb{R}$ is H\"older. In light of this, 
    \[ B^n_x = e^{-(u(f^n(x)) - u(x))/2} A^n_x,\]
    and it follows that $B$ is fiber bunched and satisfies the condition that the periodic cycle holonomy of $B$ vanishes on $us$-loops. 

    To finish, let $\varphi(x) \coloneq e^{-u(x)/2}$. If $A$ is an almost coboundary, so $A(x) = C(f(x)) \, L \, C(x)^{-1}$ for some $L \in \text{GL}(2,\mathbb{R})$, then let $\hat{C}(x) \coloneq \varphi(x) C(x)$. We see that 
    \[ B(x) = s(x)A(x) = \varphi(f(x)) C(f(x)) \,  L \,  (\varphi(x)C(x))^{-1} = \hat{C}(f(x)) \, L \, \hat{C}(x).\]
    Thus, $B$ is an almost coboundary. On the other hand, if $B$ is an almost coboundary, then we can write $B(x) = C(f(x)) \, L \, C(x)^{-1}$. Now let $\hat{C}(x) \coloneq \varphi(x)^{-1} C(x)$. We see that 
    \[ A(x) = s(x)^{-1} B(x) = \hat{C}(f(x)) \, L \, \hat{C}(x)^{-1}. \qedhere\]
\end{proof}

From now on, we work with $\text{SL}^\pm(2,\mathbb{R})$-cocycles. We now want to classify when the homoclinic holonomy group based at our fixed point $p$ is trivial. Using Lemma \ref{lem:extension}, we deduce the following.

\begin{proposition}\label{propn:KK}
    Let $A : \mathbb{T}^3 \rightarrow \text{SL}^\pm(2,\mathbb{R})$ be a H\"older continuous fiber bunched cocycle over $f$ such that $A_{p}$ is elliptic. We have that $A$ is an almost coboundary if and only the homoclinic holonomy group based at $p$ is trivial.
\end{proposition}

\begin{proof}
    Suppose that the homoclinic holonomy group based at $p$ is trivial and consider the constant cocycle $x \mapsto L_x \coloneq A_p$. Note that Lemma \ref{lem:extension} (a) is automatically satisfied with $C_p = \text{Id}$. It follows from the definition that $\rho^L_p(\gamma) = \text{Id}$ for all $\gamma \in \mathcal{H}_p$, and hence Lemma \ref{lem:extension} (b) is also satisfied. Thus, $A$ is an almost coboundary.

    On the other hand, if $A$ is an almost coboundary with $A_p$ elliptic, then $A_x = C(f(x)) L C(x)^{-1}$ for some H\"older continuous map $C : \mathbb{T}^3 \rightarrow \text{GL}(2,\mathbb{R})$. Furthermore, notice
    \[ A_p= C(p) L C(p)^{-1},\]
    so $L$ is elliptic. We also recall the relationship
    \[ A^n_x = C(f^n(x)) L C(x)^{-1}.\] 
    Now, let $\gamma = \gamma_m * \cdots * \gamma_1 \in \mathcal{H}_p$, where each $\gamma_j$ is a homoclinic $us$-loop corresponding to a point $z_j$ homoclinic to $p$. Define
    \[ R_n^j \coloneq (A_{p}^n)^{-1} A_{f^{-n}(z_j)}^{2n}(A_{p}^n)^{-1} \quad \text{and}\quad R_n \coloneq R_n^1 \, \cdots \, R_n^m,\]
    and observe that by definition
    \[ \rho^A_p(\gamma) = \lim_{n \rightarrow \infty} R_n.\]
    We can rewrite $R_n$ as 
    \[ R_n = \prod_{j=1}^mC(p) L^{-n} C(p)^{-1} C(f^n(z_j)) L^{2n} C(f^{-n}(z_j))^{-1} C(p) L^{-n} C(p)^{-1}. \]
    Using the fact that $C$ is H\"older and $z_j$ is homoclinic to $p$, we have 
    \[ C(p)^{-1} C(f^n(z_j)) \xrightarrow[]{n \rightarrow \infty} \text{Id} \quad\text{and}\quad C(f^{-n}(z_j))^{-1}C(p) \xrightarrow[]{n \rightarrow \infty} \text{Id},\]
    so leveraging the fact that $L$ is elliptic, we have $\lim_{n \rightarrow \infty} R_n = \text{Id}$.
\end{proof}

Finally, we give a useful criteria for when the homoclinic holonomy group is irreducible. 

\begin{proposition} \label{propn:irreduciblegp}
     Let $A : \mathbb{T}^3 \rightarrow \text{SL}^\pm(2,\mathbb{R})$ be a H\"older continuous fiber bunched cocycle over $f$ such that $A_{p}$ is elliptic. If $0 < |\text{Tr}(A_{p})| < 2$, then $G_{p}$ is either trivial or irreducible.
\end{proposition}

\begin{proof}
    If $G_{p}$ contains a non-trivial rotation matrix, then it must be irreducible. Thus, we may assume that $G_{p}$ does not contain a non-trivial rotation matrix and is non-trivial.
    Using Proposition \ref{propn:pcf}, it follows that for every $x \in \mathbb{T}^3$ with $f^n(x) = x$ we have 
    \[ (A_x^{n})^{-1} G_x  A_x^n = G_x.\]
    By our assumption, this implies that $G_{p}$ is invariant under conjugation by a rotation matrix which does not perserves the axes. Consequently, if $\text{Tr}(A_{p}) \neq 0$, then its Zariski closure must be the whole group, so $G_{p}$ irreducible. 
\end{proof}

\subsection{Matching trace alternative}
Given two cocycles $A$ and $B$ over $f$, we say they have \emph{matching periodic trace} if 
$ \text{Tr}(A^k_{x}) = \text{Tr}(B^k_{x}) \text{ whenever } f^k(x) = x.$ 
The main theorem of this section gives a sufficient condition for the cocycles to be conjugate in terms of the Parry representations and the matching periodic trace condition.

\begin{proposition}
\label{propn:exact-livsic}
    Let $A, B : \mathbb{T}^3 \rightarrow \text{SL}^\pm(2,\mathbb{R})$ be two H\"older continuous fiber bunched maps which have matching periodic trace. If $A_{p} = B_{p}$, $0 < |\text{Tr}(A_{p})| < 2$, and the Parry representations of $A$ and $B$ at $p$ are irreducible,
    then $A$ and $B$ are conjugate.
\end{proposition} 

In light of the previous subsection, we show that Proposition \ref{propn:exact-livsic} implies the following.

\begin{theorem}[Matching trace alternative]\label{thm:alternativetheorem}
    Let $A, B : \mathbb{T}^3 \rightarrow \text{SL}^\pm(2,\mathbb{R})$ be two H\"older continuous fiber bunched maps which have matching periodic trace. If $A_{p} = B_{p}$ and $0 < |\text{Tr}(A_{p})| < 2$, then one of the following must hold:
    \begin{enumerate}[(a)]
        \item $A$ is an almost coboundary, or
        \item $A$ and $B$ are conjugate.
    \end{enumerate}
\end{theorem}

\begin{proof}
    If $A$ is not an almost coboundary, then combining Proposition \ref{propn:KK} and \ref{propn:irreduciblegp} we must have that the Parry representations of $A$ and $B$ at $p$ are both irreducible. We can then apply Proposition \ref{propn:exact-livsic} to get that (b) holds.
\end{proof}

The goal is to reduce Proposition \ref{propn:exact-livsic} to Lemma \ref{lem:extension}. Thus, we just need to find $C_{p} \in \text{SL}(2,\mathbb{R})$ satisfying (a) and (b) in Lemma \ref{lem:extension}. The first step is to show that the matching periodic trace assumption extends to the Parry representations having the same character. 

\begin{lemma}
     Let $A$ and $B$ be two H\"older continuous $\text{SL}^\pm(2,\mathbb{R})$ fiber bunched maps over $f$. If $A$ and $B$ have matching periodic trace and $A_{p} = B_{p}$ is elliptic,
    then  $\text{Tr}(\rho^A_{p})(\gamma) = \text{Tr}(\rho^B_{p})(\gamma)$ for each $\gamma \in \mathcal{H}_{p}$.
\end{lemma}

\begin{proof}
    Let $\gamma = \gamma_m * \cdots * \gamma_1 \in \mathcal{H}_{p}$, where each $\gamma_j$ is a homoclinic $us$-loop corresponding to a point $z_j$ homoclinic to $p$. Define
    \[ R_n^j \coloneq (A_{p}^n)^{-1} A_{f^{-n}(z_j)}^{2n}(A_{p}^n)^{-1} \quad \text{and}\quad R_n \coloneq R_n^1 \, \cdots \, R_n^m.\] 
    We have 
    \[ \lim_{n \rightarrow \infty} \|\rho^A_{p}(\gamma) - R_n\| = 0.\]
    We want to approximate this with a periodic orbit. Using the improved Anosov closing lemma for multiple orbits (Proposition \ref{propn:anosov-closing}), we are able to construct a periodic point $p_n$ with period $2nm$ such that there is an $\varepsilon_n$ where for all $0 \leq k \leq 2n-1$ and $1 \leq j \leq m$ we have 
    \[ d(f^{k + 2n(j-1)}(p_n), f^{k-n}(z_j)) \leq C \alpha^{\min\{k, 2n-k\}} \varepsilon_n.\]
    Furthermore, $\varepsilon_n \rightarrow 0$ as $n \rightarrow \infty$. Let
    \[ P^j_n \coloneq (A^n_{p})^{-1} \, A^{2n}_{f^{2n(j-1)}(p_n)} \, (A^n_{p})^{-1} \quad\text{and}\quad P_n \coloneq P_n^1 \, \cdots \, P_n^m.\]
    The claim is that $P_n$ converges to $\rho^A_{p}(\gamma)$.  As a first step, we show that if we set $U_n^j \coloneq (R_n^j)^{-1} \, P_n^j$, then $\|U_n^j - \text{Id} \| \rightarrow 0$ as $n \rightarrow \infty$. Since $A$ is H\"older, the following holds for some $C_2 > 0$ and $\beta \in (0,1)$:
    \[ \|A_{f^{k + 2n(j-1)}(p_n)} - A_{f^{k-n}(z_j)}\| \leq  C_2 \alpha^{\beta \min\{k, 2n-k\}} \varepsilon_n^\beta.\]
    Now we want to approximate $R_n^j$ using this periodic orbit. Notice
    \[ \|R_n^j - P_n^j\| \leq \|R_n^j\| \cdot \|\text{Id} - A^n_{p} (A^{2n}_{f^{-n}(z_j)})^{-1} A^{2n}_{f^{2n(j-1)}(p_n)} (A^n_{p})^{-1}\|.\]
    Using the fact that $R_n^j$ converges in norm as $n$ tends to infinity and $A_{p}$ is elliptic,  there is a constant $C_3 > 0$ such that
    \[ \|R_n^j - P_n^j\| \leq C_3 \| \text{Id} - (A^{2n}_{f^{-n}(z_j)})^{-1} A^{2n}_{f^{2n(j-1)}(p_n)}\|. \]
    It now suffices to show that  $\| \text{Id} - (A^{2n}_{f^{-n}(z_j)})^{-1} A^{2n}_{f^{2nj}(p_n)}\|$ tends to zero. To that end, use a telescoping argument to write
    \[   (A^{2n}_{f^{-n}(z_j)})^{-1} A^{2n}_{f^{2n(j-1)}(p_n)} - \text{Id} = \sum_{k=0}^{2n-1} (A^{k}_{f^{-n}(z_j)})^{-1} \, r_k \, A^{k}_{f^{2n(j-1)}(p_n)},\]
    where
    \[ r_k \coloneq (A_{f^{k-n}(z_j)})^{-1} \, A_{f^{k + 2n(j-1)}(p_n)} - \text{Id}.\]
    We have that there is a constant $C_4 > 0$ such that
    \[ \|r_k\| \leq \|(A_{f^{k-n}(z_j)})^{-1}\| \cdot \|A_{f^{k-n}(z_j)} - A_{f^{k+2n(j-1)}(p_n)}\| \leq C_4 \alpha^{\beta \min \{k, 2n-k\}} \varepsilon_n^\beta.\]
    Thus, we just need to analyze $\|(A^k_{f^{-n}(z_j)})^{-1}\| \cdot \|A^k_{f^{2n(j-1)}(p_n)}\|.$
    Using the triangle inequality, we observe that there are constants $C_5, C_6 > 0$ such that
    \[\begin{split}  \|(A^k_{f^{-n}(z_j)})^{-1}\| \cdot \|A^k_{f^{2n(j-1)}(p_n)}\| &\leq \prod_{q=0}^{k-1} \|A_{f^{q-n}(z_j)}\| \cdot \|(A_{f^{q-n}(z_j)})^{-1}\| \cdot \prod_{q=0}^{k-1} \frac{\|A_{f^{2n(j-1)+q}(p_n)}\|}{\|A_{f^{q-n}(z_j)}\| } \\& \leq C_5 \theta^k \alpha^{-k \beta} \cdot \prod_{q=0}^{k-1} \left(1 + C_6 \alpha^{\min\{q, 2n-q\}\beta}\right).\end{split} \]
    Finally, we claim that the product is uniformly bounded. Indeed, using the fact that $\log(1+t) \leq t$, we have
    \[ \log \left( \prod_{q=0}^{k-1} \left(1 + C_6 \alpha^{\min\{q, 2n-q\}\beta}\right)\right) \leq \sum_{q=0}^\infty C_6 \alpha^{\beta q} < \infty. \]
    Hence, there is a $C_7 > 0$ such that 
    \[ \|(A^k_{f^{-n}(z_j)})^{-1}\| \cdot \|A^k_{f^{2n(j-1)}(p_n)}\| \leq C_7 \theta^k \alpha^{-k \beta},\]
    and consequently there is a constant $C_8 > 0$ such that
    \[ \| \text{Id} - (A^{2n}_{f^{-n}(z_j)})^{-1} A^{2n}_{f^{2n(j-1)}(p_n)}\| \leq C_8 \varepsilon_n. \]
    We can put this all together to get
    \[ \lim_{n \rightarrow \infty} \|\rho_{p}^A(\gamma_j) - P_n^j \| = 0 \quad\text{and}\quad \lim_{n \rightarrow \infty} \|U_n^j - \text{Id}\| = 0.\]

      With these two facts in mind, we can use a telescoping argument to write
    \[ P_n - R_n = \sum_{k=1}^m \left( \prod_{i=1}^{k-1} P_n^i \right) P_n^k(U_n^k - \text{Id}) \left( \prod_{i=k+1}^m  R_n^i\right),  \]
    and applying norms, we have
    \[ \|P_n - R_n\| \leq  \max_{1 \leq i \leq m} (1+\|P_n^i\|)^m  \cdot \max_{1 \leq i \leq m} (1+\|R_n^i\|)^m \cdot \sum_{k=1}^m \|P_n^k\| \cdot \|U_n^k - \text{Id}\|.  \]
    Using the fact that $\|P_n^i\|$ and $\|R_n^i\|$ are bounded in $n$ and $m$ is fixed, we can take the limit as $n$ tends to zero on both sides to deduce that $P_n \rightarrow \rho^A_{p}(\gamma)$ in norm. Since $A_{p} = B_{p}$ are elliptic, we get that along a subsequence
    \[ \text{Tr}(\rho^B_{p}(\gamma)) = \lim_{n \rightarrow \infty} \text{Tr}(B^{2nm}_{p_n}) = \lim_{n \rightarrow \infty} \text{Tr}(A^{2nm}_{p_n}) = \text{Tr}(\rho^A_{p}(\gamma)). \qedhere\]
\end{proof}

We can use \cite[Corollary XVII.3.8]{lang2012algebra} to find a matrix $C$ satisfying condition (b) in Lemma \ref{lem:extension}.

\begin{lemma}
    Let $\rho, \eta : \mathcal{H}_{p} \rightarrow \text{SL}^\pm(2,\mathbb{R})$ be irreducible representations. If ${\text{Tr}(\rho(\gamma)) = \text{Tr}(\eta(\gamma))}$ for all $\gamma \in \mathcal{H}_{p}$, then there exists $C \in \text{SL}^\pm(2,\mathbb{R})$ such that $C \rho(\gamma) C^{-1} = \eta(\gamma)$ for all $\gamma \in \mathcal{H}_{p}$.
\end{lemma}

Thus, it remains to show that (a) in Lemma \ref{lem:extension} holds. To that end, we utilize Schur's lemma, the matching trace assumption, and the fact that the matrices are in $\text{SL}^\pm(2,\mathbb{R})$.

\begin{lemma}
    If $C_{p} \, \rho^A_{p}(\gamma) \, C_{p}^{-1} = \rho^B_{p}(\gamma)$ for all $\gamma \in \mathcal{G}_{p}$, then $C_{p}  \, A_{p} \,   C_{p}^{-1} = B_{p}$.
\end{lemma}

\begin{proof}
    Using Proposition \ref{propn:holonomies} as well as the fact that $f(p) = p$, observe that 
    \[ \rho^A_{p}(f(\gamma)) = A_p \rho^A_{p}(\gamma)  A_p^{-1}.\]
    Thus, using the assumption, we deduce that for all $\gamma \in \mathcal{G}_{p}$, we have
    \[ C_{p} A_p \rho_{p}^A(\gamma) A_p^{-1} C_{p}^{-1} = B_p \rho^B_{p}(\gamma) B_p^{-1}.\]
    If we let $M_{p} := C_{p} A_p C_{p}^{-1}$, then for all $\gamma \in \mathcal{G}_{p}$ we have
    \[  B_p \rho^B_{p}(\gamma) B_p^{-1} =  C_{p} A_p \rho_{p}^A(\gamma) A_p^{-1} C_{p}^{-1} = M_{p} C_{p} \rho^A_{p}(\gamma)C_{p}^{-1} M_{p}^{-1} = M_{p} \rho^B_{p}(\gamma) M_{p}^{-1}.\]
    In particular, by Schur's lemma and Proposition \ref{propn:irreduciblegp}, we deduce that there is some $\lambda \in \mathbb{C}$ such that  
    \[ M_{p} = C_{p} A_p C_{p}^{-1} = \lambda B_p.\]
    Since $A,B \in \text{SL}^\pm(2,\R)$, we can take the determinant on both sides to deduce that $\lambda^2 = 1$. The matching trace assumption along with the fact that $0 < |\text{Tr}(A_{p})| < 2$ forces $\lambda = 1$, and the desired result follows.
\end{proof}

This finishes the proof of Proposition \ref{propn:exact-livsic}. 

\subsection{Matching theorem}

We now explain how Theorem \ref{thm:alternativetheorem} implies Theorem \ref{thm:matchingtheorem} and why we are able to use Proposition \ref{propn:KK}. Let $p_f, p_g$, and $p_L$ denote the corresponding fixed points. By structural stability, we know that the corresponding conjugacies match these fixed points. Furthermore, since $f$ and $g$ are $C^1$-close to $L$, we have that both $D^u_{p_f}f$ and $D^u_{p_g}g$ are $C^0$ close to $D^uL_{p_L}$, which is elliptic after normalizing (see Lemma \ref{lem:normalization}). 

Note that if the SRB measure of $f$ coincides with the measure of maximal entropy, then the corresponding $\text{GL}(2,\mathbb{R})$-cocycle to $f$ has cohomologically trivial determinant \cite[Proposition 4.5]{bowen1975equilibrium}. Since the periodic data of $f$ and $g$ match, the same can be said for $g$. Thus, without loss of generality, we are able to normalize these cocycles and work with $\text{SL}^\pm(2,\mathbb{R})$-cocycles. A perturbative argument shows that the normalizations of $D^u_{p_f}f$ and $D^u_{p_g}g$ are both elliptic if $D^u L_{p_L}$ is elliptic. Finally, we note that $0 < |\text{Tr}(D^uL_{p_L})| < 2$, since if $\text{Tr}(D^uL_{p_L}) = 0$, we would have that the stable eigenvalue would be an integer, and therefore zero. Again, assuming $f$ and $g$ are sufficiently close to $L$, this forces $0 < |\text{Tr}(D^u_{p_f}f)|, |\text{Tr}(D^u_{p_g}g)| < 2$, and consequently Theorem \ref{thm:matchingtheorem} holds by Theorem \ref{thm:alternativetheorem}. 

Since $D^u_{p_f}f$ and $D^u_{p_g}g$ are elliptic, we also note that we are able to use Proposition \ref{propn:KK} for such $f$ and $g$. Observe that every $us$-simple loop is also a $us$-quadrilateral by taking the ordered sequence $(a,b, a,b)$. Consequently, if for every $a,b \in \mathbb{T}^3$ in the same stable leaf we have that $\Phi_{a,b}^f$ is the identity, then by Proposition \ref{propn:KK} we have that $D^uf$ has constant periodic data, contradicting our assumption.

\bibliographystyle{alpha}
\bibliography{bibliography.bib}

\end{document}